\documentclass[10pt,a4paper]{amsproc}

\usepackage{framed}
\usepackage{graphicx}
\usepackage{float}

\usepackage[T1]{fontenc}
\usepackage[utf8]{inputenc}
\usepackage{amsmath,amssymb,amsthm,mathtools}
\usepackage[colorlinks=true,linkcolor=blue,urlcolor=blue]{hyperref}
\usepackage[margin=0.88in,footskip=0.25in]{geometry}

\numberwithin{equation}{section}
\newtheorem{theorem}{Theorem}[section]
\newtheorem{lemma}[theorem]{Lemma}
\newtheorem{proposition}[theorem]{Proposition}

\theoremstyle{definition}
\newtheorem{example}[theorem]{Example}
\theoremstyle{remark}

\DeclareMathOperator{\Ann}{ann}
\DeclareMathOperator{\Span}{span}

\title{The K\"othe conjecture via point modules}
\author{B. Greenfeld}
\author{G. King}
\author{L. Vendramin}
\address[Greenfeld]{Department of Mathematics and Statistics, CUNY Hunter College, 695 Park Avenue, New York NY 10065, USA} \email{beeri.greenfeld@hunter.cuny.edu}

\address[King]{University of Washington, Department of Mathematics, Box 354350, C-138 Padelford Hall
Seattle, WA 98195-4350 USA}
\email{gking3@uw.edu}

\address[Vendramin]{Department of Mathematics and Data Science, Vrije Universiteit Brussel, Pleinlaan 2, 1050 Brussel, Belgium}
\email{Leandro.Vendramin@vub.be}

\begin{document}

\begin{abstract}
We give a short construction refuting the K\"othe Conjecture.
The construction is based on the point module approach to the problem, 
introduced by the first-named author, Easton, Ford, and King.
\end{abstract}

\maketitle

\section{Introduction and methodology}

The K\"othe conjecture asserts that the sum of two nil left ideals of a ring is
again nil; by a theorem of Krempa~\cite{MR306251} and Amitsur~\cite{MR78345}, this is equivalent to the
assertion that $A[t]$ is Jacobson radical for every nil ring $A$. These and other equivalent 
formulations are surveyed in \cite{MR1879880}. 
A counterexample was recently found autonomously by a pre-release version of
GPT-6 Astra, running in Epoch AI's \texttt{LeanOpenProblems} harness without
human steering, and verified in Lean 4 \cite{koethe-repo}; a human-readable
account appeared shortly afterwards in \cite{ABM26}. An important antecedent to these constructions is
Smoktunowicz's groundbreaking work~\cite{MR1793911,MR2471940} on nil algebras with
non-nil polynomial extensions. The methods produced by Smoktunowicz 
have become fundamental in successive work in the field. 

Earlier, Easton, Ford, Greenfeld and King \cite{MR4976689} had proposed an
approach to the K\"othe conjecture through point modules. 
They
observed the obstruction that point modules pose to the conjecture, and
constructed graded-nil algebras admitting point modules. The paper \cite{MR4976689}
provides the foundations and the strategy  for the present note. 
The algebras constructed in \cite{MR4976689} are graded-nil, which is not sufficient to 
contradict the conjecture; the content of Theorem~\ref{thm:main} is that 
graded \emph{and} nil can be achieved. By Proposition \ref{prop:Koethe}, 
for such an algebra $A$, the polynomial ring $A[t]$ is not
Jacobson radical, hence a counterexample to the K\"othe conjecture. 

\begin{theorem}
\label{thm:main}
There exists a finitely generated, graded, nil algebra, admitting a point module, and therefore refuting the K\"othe Conjecture.
\end{theorem}

The paper \cite{MR4976689} was inspired by Small and
Zelmanov \cite{MR2273992}, who first considered point modules over nil algebras. 
Rowen \cite{MR1029310} proposed an infinite-matrix approach to the K\"othe conjecture. 
It was later observed that these matrices
generate a monomial algebra \cite{BG17}; their action on infinite row vectors gives a point module.

Loosely speaking, to construct a nil algebra with a point module, one must produce
an infinite sequence of points in a projective space on which a system of
multilinear equations vanishes, each equation depending on finitely many
consecutive terms of the sequence; this is essentially what is called a 
`multilinearly recurrent' sequence in \cite{MR4976689}. We build the sequence `in windows': a
solution of the first equation is placed in every window of some fixed length, a
solution of the second in every window of a larger length, and so on. 
This is analogous to the dyadic construction in \cite[Section~3.2]{MR4976689} (see also Lemma 3.3 therein); however, in
\cite{MR4976689} only powers of homogeneous polynomials occur, and the resulting system
is easier to solve: it suffices to place a solution in every window of bounded
length.

The main new ingredient is that the system of multilinear equations can be solved 
using an inductive dimension computation. Specifically, 
a batch of previously prescribed points lowers the dimension by at most its length, 
so the zero locus is non-empty once the free factors exceed the prescribed points 
(see~Lemma~\ref{lem:dimension}). 
Its proof is an elementary Segre-type argument. Arranging the positions of the prescribed points is then done using combinatorial machinery in Lemma~\ref{lem:combinatorics}.

\subsection*{Use of AI}

The main argument proving Lemma~\ref{lem:dimension} was suggested to us by GPT-6 Astra. This lemma is analogous to the part in the proof of \cite[Theorem~3.4]{MR4976689} where the authors prove `multilinear recurrence.' 
GPT-6 Astra also suggested an earlier version of sparsity of the set of prescribed points, which we eventually developed into the current notion of a `good prescription.' It also helped us  complete the proof of the second case in Lemma~\ref{lem:combinatorics}.

The same model observed 
that it suffices for segments of length $\deg(f)$ of homogeneous elements of the
point module, with bounded gaps between them, to be annihilated by a power of
$f$, in order for a power of $f$ to annihilate the whole module (see the argument
following Step $D-1$ in Section~\ref{sec:construct}). It also proposed the natural (and necessary) extension of the argument in \cite[Proposition~6.1]{MR4976689} to Proposition \ref{prop:Koethe} in the current paper.

All of the arguments proposed by AI have been checked and rewritten by the authors, who take full responsibility for the contents of this paper.

\section{Strategy}

Our goal is to construct a positively graded nil algebra, finitely generated in degree one, which admits a point module, namely, a graded module generated in degree zero, all of whose homogeneous components are one-dimensional. Every such algebra is a counterexample to the K\"othe conjecture. In \cite{MR4976689}, we did it for graded-nil algebras: these are algebras in which every homogeneous element is nilpotent. 

Let $\Bbbk$ be a countable, algebraically closed field. Let $\mathcal{F} = \Bbbk\langle x,y,z\rangle^{+}$ be the (non-unital) free $\Bbbk$-algebra. For $f\in \mathcal{F}$, write $f=f_1+\dots+f_D$ with each $f_i$ homogeneous of degree $i$.
As in \cite{MR4976689}, we start with a point module over the free algebra
\[
M = \Bbbk e_0 \oplus \Bbbk e_1 \oplus \Bbbk e_2 \oplus \dots
\]
Such a point module corresponds to a sequence
\[
\mathbf{p} = (\mathbf{p}_0,\mathbf{p}_1,\dots) \in \left(\mathbb{P}_\Bbbk^2\right)^{\infty}
\]
where each $\mathbf{p}_i=[a_i\colon b_i\colon c_i]$ and the action of $\mathcal{F}$ on $M$ is given by $e_i \cdot x = a_i e_{i+1}, e_i \cdot y = b_i e_{i+1}, e_i \cdot z = c_i e_{i+1}$.
Given a homogeneous element $f\in \mathcal{F}$ and $j\geq 0$, its \emph{$j$-th multilinearization} $f^{(j)}$ is a multilinear function on $\left(\mathbb{P}_\Bbbk^2\right)^{\infty}$ given by (i) replacing each $x$ by $a$, $y$ by $b$, $z$ by $c$, and (ii) indexing each $a,b,c$ incrementally starting from $j$. This is best demonstrated by an example. If $f=xy+z^2$ then $f^{(0)}=a_0 b_1+c_0 c_1, f^{(1)}=a_1 b_2 + c_1 c_2$, etc.
Note that $e_j\cdot f = f^{(j)}(\mathbf{p}) e_{j+\deg(f)}$. Hence, if $f^{(j)}(\mathbf{p})=0$ for all $j\geq 0$ then $M$ becomes a point module over $\mathcal{F}/\langle f \rangle$. 

Our strategy is therefore as follows: enumerate all (not necessarily homogeneous) elements in $\mathcal{F}$, say, $f_1,f_2,\dots$, and for each $i$, suppose $f_i$ is supported on degrees $[1,D_i]$; choose $N_i\gg 1$ such that when writing $g_i := f_i^{N_i} = g_{i,N_i}+\dots+g_{i,D_i N_i}$, all the multilinearizations $g_{i,t}^{(j)},\ N_i\leq t\leq D_i N_i,\ j\geq 0$, vanish on $\mathbf{p}$. The challenge, of course, is to construct such $\mathbf{p}$.

\section{Combinatorics on words}

% Transcription of handwritten page 1. No indexing or notation repairs.
We are trying to construct a sequence
\[
 \mathbf{p}_0,\mathbf{p}_1,\mathbf{p}_2,\dots
 \quad\in\left(\mathbb{P}_{\Bbbk}^2\right)^\infty.
\]

A \emph{prescription} is a choice of some of the points in the
sequence. A point that was not prescribed is \emph{free}.
We denote by $a_0$ the number of the first batch of prescribed coordinates, then by $b_1$ the length
of the next free batch, then by $a_1$ the length of the prescribed batch following it, etc.

A \emph{good prescription} is one for which any sequence $\{d_i\}_{i=a_0}^{\infty}$ following the rules below has $d_i\geq 1$ for all $i$: we require $d_{a_0} \geq 1$ (first free index); that if $i+1$ is a free index then $d_{i+1}\geq d_i+1$; and that if $i$ is free and $i+1$ is prescribed (in the batch of length $a_j$), then
$d_{i+r}\geq d_i-r$ for all $1\leq r\leq a_j$.

While this definition seems a bit implicit, it has the following convenient equivalent form; nevertheless, we prefer the definition above as we will use it as is in the construction below. A prescription is good if and only if 
\[
 b_1-a_1\geq1;\qquad b_1+b_2-a_1-a_2\geq1;\qquad\text{etc.}
\]

\begin{lemma}
\label{lem:combinatorics}
Given $L_1,L_2,\dots$ powers of $2$, we can find $P_1,P_2,\dots$ such that: (i) $P_i$ depends on $L_1,\dots,L_i$ and on $P_1,\dots,P_{i-1}$, (ii) $P_i$ is a power of $2$ and can be enlarged if needed, and (iii) every prescription in indices
\[
\mathcal{P}_k := \{i\in\mathbb{Z}_{\geq 0}\mid i\bmod P_1<L_1\}\cup \dots \cup \{i\in\mathbb{Z}_{\geq 0}\mid i\bmod P_k<L_k\}
\]
is good, for every $k$.
\end{lemma} 

\begin{proof}
We may assume that 
\[
P_k \geq P_{k-1}, 4(L_1+\dots+L_k),2^{k+10}L_k
\]
for all $k$. We claim that this choice of $P_1,P_2,\dots$ makes $\mathcal{P}_k$ good, for every $k$. We work by induction, assuming this claim holds for all $i\leq k$,  aiming to prove it for $k+1$. (The base case is straightforward.) Let $a$ be the first free index with respect to $\mathcal{P}_{k+1}$ (that is, the smallest non-negative integer outside it). 
Notice that our choice makes $\sum_{i=1}^{\infty} \frac{L_i}{P_i} < 0.01$. 
Let $L=\max\{L_1,\dots,L_{k+1}\}$.
Now given any $N\geq a$, we need to count how many prescribed indices appear within $[a,N]$; let $p(N):=|\mathcal{P}_{k+1}\cap [0,N]|$.

Notice that all initial blocks of prescribed points are contained in $[0,L-1]$. Beyond that, we have $L_j$ more consecutive prescribed points at the starting positions $P_j,2P_j,\dots$; hence there are at most $\lfloor \frac{N}{P_j} \rfloor$ such blocks that start at or before $N$. Hence by a union bound, 
\[
p(N)\leq L+L_1 \left\lfloor \frac{N}{P_1} \right\rfloor+\dots+L_{k+1} \left\lfloor \frac{N}{P_{k+1}} \right\rfloor \leq L + 0.01 N.
\]
In particular, $p(2L)\leq 1.02 L < 2L$, so the first free index $a$ must appear before $2L$. 

Now if $N\geq 2L$, the number of prescribed indices in $[a,N]$ is at most $p(N)-L\leq 0.01 N$, so the number of free points in $[a,N]$, which is at least $N-p(N)\geq 0.99 N - L \geq 0.49 N$, is larger than the number of prescribed indices in $[a,N]$.

To complete the picture, suppose that $a<N<2L$ ($N=a$ is obvious, since $a$ is the first free index).
Suppose that $r$ is the last index for which $P_1,\dots,P_r\leq L$ (or formally $0$ if even $P_1>L$); recall that $P_1\leq\dots\leq P_{k+1}$ and $P_{k+1}>L$, so $r\leq k$. If $j>r$ then $P_j>L$ so $P_j\geq 2L$ (recall these are all powers of two), and so 
\begin{equation} \label{eq:4}
\mathcal{P}_{k+1}\cap [L,2L-1] = \mathcal{P}_r\cap [L,2L-1]. 
\end{equation}
But for each $j\leq r$, since $P_j$ divides $L$ (being a smaller power of two), 
\begin{equation} \label{eq:5}
x\in \mathcal{P}_r\cap [L,2L-1]\ \ \iff\ \ x-L\in \left(\mathcal{P}_r \cap [0,L-1]\right).
\end{equation}
In particular, the first free index of $\mathcal{P}_r$ is $a-L$. 
By the induction hypothesis, 
\[
|\mathcal{P}_r\cap [a-L,N-L]| < |[a-L,N-L]\setminus \mathcal{P}_r|
\]
so by \eqref{eq:5}, 
$|\mathcal{P}_r\cap [a,N]| < |[a,N]\setminus \mathcal{P}_r|$. By \eqref{eq:4}, we can replace $\mathcal{P}_r$ by $\mathcal{P}_{k+1}$ in the last inequality, completing the proof.
\end{proof}

\section{Solving multilinear equations}

Suppose we have a prescription $\mathcal{P}$ where the first $a_0$ points $\mathbf{p}_0,\dots,\mathbf{p}_{a_0-1}$ are prescribed, the next $b_1$ points $\mathbf{p}_{a_0},\dots,\mathbf{p}_{a_0+b_1-1}$ are free, the $a_1$ following them are prescribed, and so on. 
Suppose that we are given multilinear equations
\[
 F_t(\mathbf{p}_0,\dots,\mathbf{p}_t)
\]
for $a_0\leq t\leq M$. By this we mean that $F_t$ is a linear combination of monomials of the form
$X_0X_1\cdots X_t$ with each $X_j\in\{a_j,b_j,c_j\}$. For $j$ a prescribed index, we replace $X_j$ by the value of the prescribed point $\mathbf{p}_j$. Thus $F_t$ becomes a function on
\[
(\mathbb{P}^2)^{\#\text{free points }\leq t}=\left(\mathbb{P}^2\right)^{|[0,t]\setminus \mathcal{P}|}
\]
and, in fact, under its Segre
embedding into some $\mathbb{P}^N$, a \emph{linear} equation (as all free factors are involved).

\begin{example} 
Suppose $\mathbf{p}_0,\mathbf{p}_2,\mathbf{p}_4,\dots$
are prescribed. Our equations take the forms
\[
 F_2=\alpha a_1+\beta b_1+\gamma c_1,\ \ \ 
 F_3=\alpha a_1a_3+\alpha'a_1b_3+\alpha''a_1c_3+
 \beta b_1a_3+\beta'b_1b_3+\beta''b_1c_3+\dots
\]
\end{example}

\begin{lemma}
\label{lem:dimension}
If the prescription $\mathcal{P}$ is good, we can find
points $\mathbf{p}_0,\dots,\mathbf{p}_M$, agreeing with those already prescribed, such that $F_{a_0}(\mathbf{p})=\dots=F_M(\mathbf{p})=0$.
\end{lemma} 

\begin{proof}
For each $i\geq a_0$, let $V_i \subseteq \left(\mathbb{P}^2\right)^{i+1}$ be the zero locus of $F_{a_0},\dots,F_i$ and all prescribed points among $\mathbf{p}_0,\dots,\mathbf{p}_i$.
For instance,
\[
 V_{a_0}\subseteq
 \underset{\mathbf{p}_0}{\mathbb{P}^2}\times\cdots\times
 \underset{\mathbf{p}_{a_0-1}}{\mathbb{P}^2}\times
 \underset{\mathbf{p}_{a_0}}{\mathbb{P}^2}
\]
is a linear subspace of
$\{\mathbf{p}_0\}\times\cdots\times\{\mathbf{p}_{a_0-1}\}\times\mathbb{P}^2$,
as it is given by a single linear equation on
$\mathbf{p}_{a_0}$. 
Thus $\dim V_{a_0}\geq1$. 
Now if $i+1$ is a free index ($i+1\not\in \mathcal{P}$) then
\[
 \dim V_{i+1}\geq\dim V_i+\dim\mathbb{P}^2-1=\dim V_i+1
\]
since $V_{i+1}$ is the intersection of $V_i \times \mathbb{P}^2$ with a single hypersurface given by a multilinear equation involving all factors so far.
And if $i$ is free and $i+1,\dots,i+a_j$ are a  prescribed batch then
\[
 V_{i+a_j}=V_i\times\{\mathbf{p}_{i+1}\}\times\cdots\times
 \{\mathbf{p}_{i+a_j}\}\cap
 \{F_{i+1}=\cdots=F_{i+a_j}=0\},
\]
but each $F_{i+1},\dots,F_{i+a_j}$ is a multilinear equation on all free points so far, involving coordinates from \emph{all} free factors; each
one of them thus becomes a \emph{linear} equation under the Segre
embedding, as discussed above.
Hence
\[
 \dim V_{i+a_j}\geq\dim V_i-a_j.
\]
By the goodness of the prescription, $\dim V_t\geq 1$ for all $t\geq a_0$, so we can take any $(\mathbf{p}_0,\dots,\mathbf{p}_M)\in V_M$.
\end{proof}

\begin{center}
\begin{figure}[H]
\includegraphics[width=0.65\linewidth]{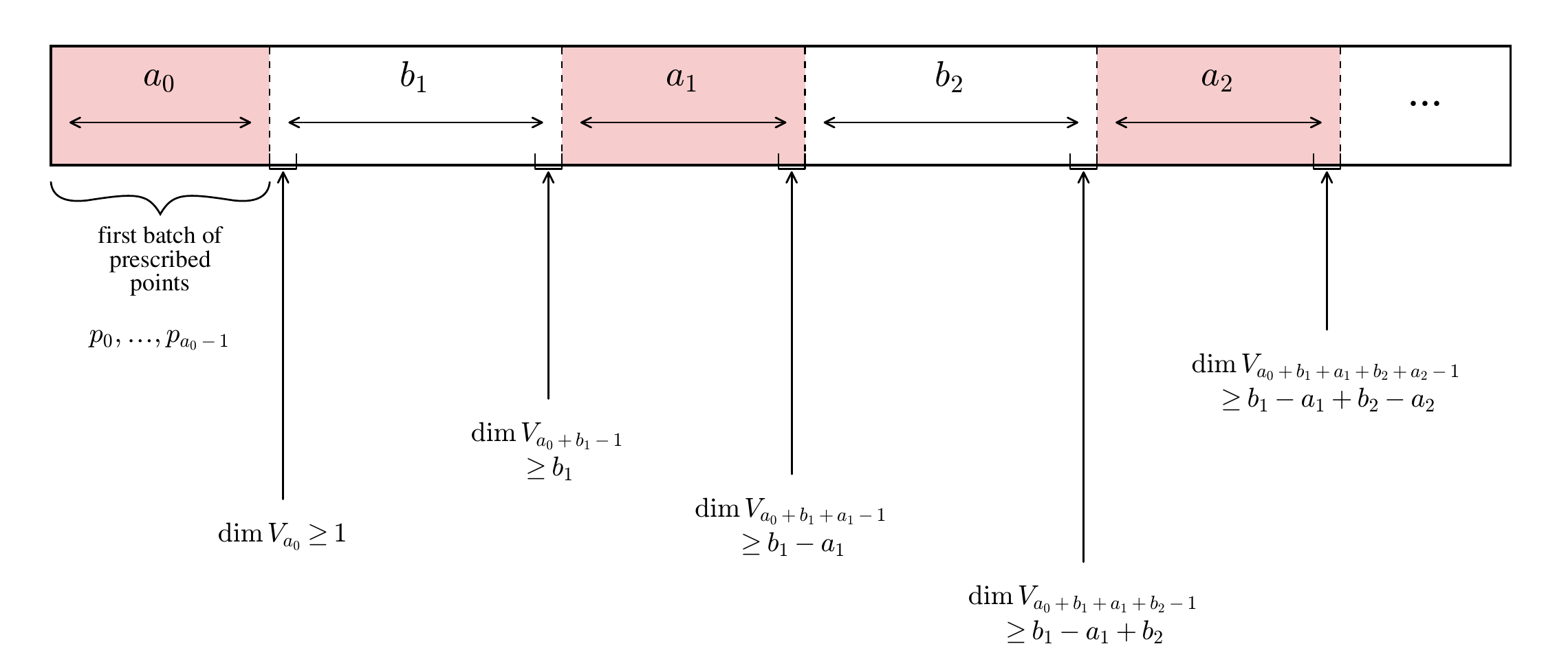}
\caption{An illustration of Lemma~\ref{lem:dimension}.}
\end{figure}
\end{center}

\section{Constructing \texorpdfstring{$\mathbf{p}$}{p}} 
\label{sec:construct}

We are now ready to construct $\mathbf{p}=(\mathbf{p}_0,\mathbf{p}_1,\dots)$
inductively. Suppose that we are given a good prescription where prescribed
points are in indices
\[
 \{i\in\mathbb{Z}_{\geq0}\mid i\bmod P_1<L_1\}\cup \dots \cup \{i\in\mathbb{Z}_{\geq0}\mid i\bmod P_k<L_k\}.
\]

Let $f\in\Bbbk\langle x,y,z\rangle^+$ be given, $f$ supported on degrees
$[1,D]$. We will apply Lemma~\ref{lem:dimension}, $D$ times. We may assume that the first $D$ indices are prescribed (if necessary, apply first Lemma~\ref{lem:combinatorics} by enlarging the collection of prescribed points).

\emph{Step 0:} Suppose $a$ is the first free index; we take $N_0=a+1$. Consider
$f^{N_0}=g_{N_0}+\cdots+g_{DN_0}$ (each $g_j$ is homogeneous of degree $j$) and let
\[
F_{a}(\mathbf{p}_0,\dots,\mathbf{p}_{a})=g_{N_0}^{(0)}(\mathbf{p}_0,\dots,\mathbf{p}_{a}),\ \dots,\ F_{DN_0-1}(\mathbf{p}_0,\dots,\mathbf{p}_{D N_0-1})=g_{DN_0}^{(0)}(\mathbf{p}_0,\dots,\mathbf{p}_{D N_0-1})
\]
be their $0$-th multilinearizations. 
By Lemma~\ref{lem:dimension}, we can refine the prescription by further
prescribing $\mathbf{p}_0,\dots,\mathbf{p}_{D N_0 - 1}$, such that
$F_{a}=\cdots=F_{DN_0-1}=0$.

Take $L_{k+1}$ a power of two such that $L_{k+1} \geq DN_0$. 
By Lemma~\ref{lem:combinatorics}, we can take $P_{k+1}$ very large, replicate
$\mathbf{p}_0,\dots,\mathbf{p}_{L_{k+1}-1}$ every $P_{k+1}$ points (that is, prescribe $\mathbf{p}_{j+sP_{k+1}}=\mathbf{p}_j$, $0\leq j\leq L_{k+1}-1$, $s\geq 1$)
and this will still be a good prescription. Notice that, at the level of the module $M$,
\[
 e_0\cdot f^{N_0}=e_{P_{k+1}}\cdot f^{N_0}
 =e_{2P_{k+1}}\cdot f^{N_0}=\cdots=0.
\]

\emph{Step $1\leq t\leq D-1$:} By now, as a result of the previous step, we have a prescription
\[
 \{i\in\mathbb{Z}_{\geq0}\mid
 i\bmod P_1<L_1\} \cup \dots \cup \{i\in\mathbb{Z}_{\geq0}\mid i\bmod P_{k+t}<L_{k+t}\}.
\]
Let $a$ be the index of the first free point and take $N_t=a+1-t$; write $f^{N_t}=g_{N_t}+\cdots+g_{DN_t}$, a sum of homogeneous components 
(not to be confused with the $g_j$'s from Step 0). We let
\begin{align*}
F_{\underbrace{N_t+t-1}_{=a}}(\mathbf{p}_t,\dots,\mathbf{p}_{\underbrace{N_t+t-1}_{=a}}) & =g_{N_t}^{(t)}(\mathbf{p}_t,\dots,\mathbf{p}_{\underbrace{N_t+t-1}_{=a}}) \\ & \vdots  \\ F_{DN_t+t-1}(\mathbf{p}_t,\dots,\mathbf{p}_{D N_t+t-1}) & =g_{DN_t}^{(t)}(\mathbf{p}_t,\dots,\mathbf{p}_{D N_t+t-1})
\end{align*}
the $t$-th multilinearizations of the homogeneous terms.

By Lemma~\ref{lem:dimension}, we can refine the prescription by further
prescribing $\mathbf{p}_0,\dots,\mathbf{p}_{D N_t + t - 1}$, such that
$F_{a}=\cdots=F_{DN_t + t -1}=0$.

Take $L_{k+t+1}$ a power of two such that $L_{k+t+1} \geq DN_t+t$. 
By Lemma~\ref{lem:combinatorics}, we can take $P_{k+t+1}$ very large, replicate
$\mathbf{p}_0,\dots,\mathbf{p}_{L_{k+t+1}-1}$ every $P_{k+t+1}$ points (that is, prescribe $\mathbf{p}_{j+sP_{k+t+1}}=\mathbf{p}_j$, $0\leq j\leq L_{k+t+1}-1$, $s\geq 1$)
and this will still be a good prescription. 
At the level of the module $M$, we just ensured that
\[
 e_t\cdot f^{N_t}=e_{t+P_{k+t+1}}\cdot f^{N_t}
 =e_{t+2P_{k+t+1}}\cdot f^{N_t}=\cdots=0.
\]
After Step $D-1$, we have a good prescription and:
\begin{equation} \label{eq:1}
e_{t+sP_{k+D}}\cdot f^{\max\{N_0,\dots,N_{D-1}\}}  = 0 \ \text{for all}\  0\leq t\leq D-1,\  s \geq 0\ \text{(by Step $t$)}
\end{equation}
% \[
%  \begin{aligned}
%  e_{sP_{k+1}}\cdot f^{N_0}&=0&&\forall\ s\geq0\quad\text{(by Step 0)}\\
%  e_{1+sP_{k+2}}\cdot f^{N_1}&=0&&\forall\ s\geq0\quad\text{(by Step 1)}\\
%  &\vdots&&\\
%  e_{D-1+sP_{k+D}}\cdot f^{N_{D-1}}&=0&&\forall\ s\geq0\quad\text{(by Step $D-1$)}.
%  \end{aligned}
% \]
Take $P:=\max\{P_1,\dots,P_{k+D}\}$ (recall they are all powers of $2$; and $P\geq D$ as the first $D$ points are assumed to be prescribed) and $N:=\max\{N_0,\dots,N_{D-1}\}$, by \eqref{eq:1}:
\begin{equation} \label{eq:2}
 e_{t+sP}\cdot f^N=0\qquad\forall\ 0\leq t\leq D-1,\ s\geq0.
\end{equation}
The next (and final) step is to show that for \emph{all} $t\geq0$, $e_t\cdot f^{N+P-D}=0$.
We know it for $t\in [sP,sP+D-1]$ ($s\geq 0$) and need to prove for all $sP+D\leq t\leq(s+1)P-1$. Now
\begin{equation} \label{eq:3}
 e_{(s+1)P-1}\cdot f
 \in\Span_{\Bbbk}\{e_{(s+1)P},e_{(s+1)P+1},\dots,e_{(s+1)P+D-1}\} \overset{\eqref{eq:2}}{\subseteq} \Ann(f^N)
\end{equation}
so $e_{(s+1)P-1}\in\Ann(f^{N+1})$. Next,
\[
 \begin{aligned}
 e_{(s+1)P-2}\cdot f
 \in\Span_{\Bbbk}\{e_{(s+1)P-1},e_{(s+1)P},\dots,e_{(s+1)P+D-2}\} \overset{\eqref{eq:2},\eqref{eq:3}}{\subseteq} \Ann(f^{N+1})
 \end{aligned}
\]
so $e_{(s+1)P-2}\in\Ann(f^{N+2})$.
In this manner, eventually,
\[
e_{sP+D} = e_{(s+1)P-(P-D)} \in\Ann(f^{N+P-D}).
\]
We conclude that some power of $f$ annihilates the entire point module. % {\color{red}Enumerate $\mathcal F$ and repeat the construction for every element.
% At each substep the prescribed prefix extends beyond the previous
% first free index, since $DN_t+t\geq a+1$. Thus every coordinate
% is eventually fixed. Earlier prescriptions remain unchanged,
% so all previously obtained annihilation identities persist.

% Set $I=\Ann_{\mathcal F}(M)$ and $A=\mathcal F/I$.
% The two-sided ideal $I$ is homogeneous: if
% $h=\sum_d h_d\in I$, with each $h_d$ homogeneous, then
% $0=e_jh=\sum_d e_jh_d$ has summands in distinct degrees,
% so every $h_d$ annihilates $M$. Moreover, for every
% $f\in\mathcal F$, the construction gives an exponent $E$
% such that $Mf^E=0$, hence $f^E\in I$. Therefore $A$ is
% nil, positively graded, and generated in degree one by the
% images of $x,y,z$. Since each $\mathbf p_j$ has a nonzero
% coordinate, $M_jA_1=M_{j+1}$ for every $j$, so $M$ is
% a point module over $A$.}
Repeating this with all $f\in\Bbbk\langle x,y,z\rangle^+$
(arbitrarily enumerated), we obtain a nil, graded quotient of $\Bbbk\langle x,y,z\rangle^+$, 
admitting a point module.

\section{The K\"othe Conjecture}

To complete the picture, we need the following argument, which appeared in \cite{MR4976689} in the case of two generators.

\begin{proposition} \label{prop:Koethe}
Let $A$ be a positively graded algebra, finitely generated in degree one. Suppose that $A$ is nil and admits a point module. Then $A[t]$ is not Jacobson radical.
\end{proposition}

\begin{proof}
Let $M=\bigoplus_{i=0}^{\infty} \Bbbk e_i$ be a point module over $A$, and suppose that $e_i \cdot x_j = c_{i,j} e_{i+1}$.
Suppose that $A[t]$ is Jacobson radical. Extend the grading from $A$ to $A[t]$ by $\deg(t)=0$. Let $x_1,\dots,x_m$ be degree-one generators. Recall that the Jacobson radical of an $\mathbb{N}$-graded algebra is graded-nil \cite{Bergman}, so $x_1+t x_2+\dots+t^{m-1} x_m$ is nilpotent, say, of index $N$.

% Extend scalars to some infinite field extension $\mathbb{K}/\Bbbk$, $\tilde{M}=\bigoplus_{i=0}^{\infty} \mathbb{K} e_i$ an $\tilde{A}:=\mathbb{K}\otimes_\Bbbk A$-module. 
Hence for every $\lambda\in \Bbbk$, we have $(x_1+ \lambda x_2+\dots+\lambda^{m-1} x_m)^N=0$, so, acting with it on $e_0$,
\[
(c_{0,1}+ \lambda c_{0,2}+\dots+\lambda^{m-1} c_{0,m})(c_{1,1}+ \lambda c_{1,2}+\dots+\lambda^{m-1} c_{1,m})\dots (c_{N-1,1}+ \lambda c_{N-1,2}+\dots+\lambda^{m-1} c_{N-1,m}) = 0
\]
and hence, since $\Bbbk$ is infinite, there is some $0\leq j\leq N-1$ such that for infinitely many $\lambda$'s, 
\[
c_{j,1}+ \lambda c_{j,2}+\dots+\lambda^{m-1} c_{j,m} = 0.
\]
By a Vandermonde argument (using any distinct $m$ scalars $\lambda_1,\dots,\lambda_m$), we see that $c_{j,1}=\dots=c_{j,m}=0$, contradicting that $M$ is a point module: it cannot be generated in degree zero, as $e_j\cdot x_1=\dots=e_j\cdot x_m=0$ so $e_j\cdot A=0$ and thus $e_{j+1}\notin e_0\cdot A$, a contradiction.
\end{proof}

%\begin{thebibliography}{99}

% \bibitem{ABM26}
% Adamczewski...

% \bibitem{Bergman}
% Jacobson radical is graded-nil

% \bibitem{MR4976689}
% Easton, Ford, Greenfeld, King

% \bibitem{MR2273992}
% On point modules

% \bibitem{Smoktunowicz}
% Smoktunowicz, polynomial rings over nil rings

%\end{thebibliography}

\bibliographystyle{abbrv}
\bibliography{refs}

\end{document}